\documentclass[11pt, reqno]{amsart}
\usepackage{eucal}

\usepackage{graphicx} 
\usepackage{comment}
\usepackage{color}
\usepackage{amsmath}
\numberwithin{equation}{section}
\usepackage{enumerate}

\usepackage[colorlinks, linkcolor=black, citecolor=blue, linktocpage,urlcolor=blue]{hyperref}
\theoremstyle{definition}
\newtheorem{thm}{Theorem}[section]
\newtheorem{definition}[thm]{Definition}
\newtheorem{proposition}[thm]{Proposition}

\newtheorem{lemma}[thm]{Lemma}
\newtheorem{corollary}[thm]{Corollary}

\newtheorem{remark}[thm]{Remark}

\author{Haiping Fu}
\address{Department of Mathematics, Nanchang University, Nanchang 330031, People’s Republic of China}
\email{mathfu@126.com}
\thanks{Supported in part by National Natural Science Foundations of China \#12461008 and 12271069,  Jiangxi Province
Natural Science Foundation of China \#20202ACB201001, Jiangxi Province Graduate Student Innovation Special Fund Project \#YC2025-B037.}
\author{Yao Lu}
\address{Department of Mathematics, Nanchang University, Nanchang 330031, People’s Republic of China}
\email{luyao@email.ncu.edu.cn}
\author{Zhilin Dai}
\address{School of Science, East China University of Technology, Fuzhou, 344000, Jiangxi, China}
\email{202361014@ecut.edu.cn}

\usepackage{fancyhdr}


\makeatletter

\begin{document}
		\newcommand{\Ext}{\bigwedge\nolimits}
	\newcommand{\Div}{\operatorname{div}}
	\newcommand{\Hol} {\operatorname{Hol}}
	\newcommand{\diam} {\operatorname{diam}}
	\newcommand{\Scal} {\operatorname{Scal}}
	\newcommand{\scal} {\operatorname{scal}}
	\newcommand{\Ric} {\operatorname{Ric}}
	\newcommand{\Hess} {\operatorname{Hess}}
	\newcommand{\grad} {\operatorname{grad}}
	\newcommand{\Rm} {\operatorname{Rm}}
	\newcommand{ \Rmzero } {\mathring{\Rm}}
	\newcommand{\Rc} {\operatorname{Rc}}
	\newcommand{\Curv} {S_{B}^{2}\left( \mathfrak{so}(n) \right) }
	\newcommand{ \tr } {\operatorname{tr}}
	\newcommand{ \Riczero } {\mathring{\Ric}}
	\newcommand{ \Ad } {\operatorname{Ad}}
	\newcommand{ \dist } {\operatorname{dist}}
	\newcommand{ \rank } {\operatorname{rank}}
	\newcommand{\Vol}{\operatorname{Vol}}
	\newcommand{\dVol}{\operatorname{dVol}}
	\newcommand{ \zitieren }[1]{ \hspace{-3mm} \cite{#1}}
	\newcommand{ \pr }{\operatorname{pr}}
	\newcommand{\diag}{\operatorname{diag}}
	\newcommand{\Lagr}{\mathcal{L}}
	\newcommand{\av}{\operatorname{av}}
	\newcommand{ \floor }[1]{ \lfloor #1 \rfloor }
	\newcommand{ \ceil }[1]{ \lceil #1 \rceil }
	\newcommand{\Sym} {\operatorname{Sym}}
	\newcommand{\bcirc}{ \ \bar{\circ} \ }
	\newcommand{\sign}[1]{\operatorname{sign}(#1)}
	\newcommand{\cone}{\operatorname{cone}}
	\newcommand{\pbd}{\varphi_{bar}^{\delta}}
	\newcommand{\End}{\operatorname{End}}
	
	\renewcommand{\labelenumi}{(\alph{enumi})}
	\newtheorem{maintheorem}{Theorem}[]
	\renewcommand*{\themaintheorem}{\Alph{maintheorem}}
	\newtheorem*{remark*}{Remark}

	\vspace*{-1cm}

	\title{Manifolds with harmonic curvature and curvature operator of the second kind}
	\begin{abstract}
	 We prove that complete Riemannian manifolds of dimension $n\ge3$ with harmonic curvature and $\frac{n(n+2)}{2(n+1)}$-nonnegative curvature operator of the second kind must be Einstein. In particular, We show that complete Einstein manifolds of dimension $n\ge4$ with  $\frac{3n(n-1)^2(n+2)}{2(5n^3+3n^2-30n+16)}$-nonnegative curvature operator of the second kind must be of constant curvature, which generalizes the work of Dai-Fu \cite{DF}.
	\end{abstract}
	\maketitle
	
	\pagestyle{fancy}
	\fancyhead{} 
	\fancyhead[CE]{}
	\fancyhead[CO]{}
	\fancyfoot{}
	\fancyfoot[CE,CO]{\thepage}
	\setlength{\footskip}{13.0pt}
	\section{introduction}
	It is an important topic in geometry to understand how geometric assumptions restrict the topology of the underlying Riemannian manifold. In this direction, Tachibana \cite{Tac} showed that any compact Riemannian manifold of dimension $n\ge3$ with harmonic curvature and positive curvature operator  is isometric to a quotient of  the standard sphere. If the curvature operator is only nonnegative, then the manifold is locally symmetric. Later, Tran \cite{tran} proved that  a compact Riemannian manifold $M$ of dimension $n\ge 4$ with harmonic Weyl curvature and  positive curvature operator, then $M$ is locally conformally flat. In \cite{PW1}  Petersen and Wink showed that a compact manifold with harmonic Weyl tensor and $[\frac{n-1}{2}]$-nonnegative curvature operator is either globally conformal to a quotient of  the standard sphere or locally symmetric, and that the first possibility always occurs if the curvature operator is $[\frac{n-1}{2}]$-positive. In \cite{PW} they proved a similar Tachibana-type theorem under the stronger condition that the manifold be Einstein.
More recently, Colombo, Mariani, and Rigoli \cite{CMR} established that a compact Riemannian manifold of dimension $n\ge3$ with harmonic curvature and  $[\frac{n-1}{2}]$-positive curvature operator  must have constant sectional curvature, thereby generalizing Tachibana's result. This result was improved by Bettiol and Goodman \cite{BG} to $\frac{n-1}{2}$-positive curvature operator for $n\geq5$ and $\frac{n}{2}$-positive curvature operator for $n=3,4$.
	
	For curvature operator of the second kind, Nishikawa \cite{Nis} conjectured that  a closed simply connected Riemannian manifold for which the curvature operator of the second kind is nonnegative and positive is diffeomorphic to a Riemannian locally symmetric space and a spherical space form, respectively. This conjecture has been proved and improved by Cao–Gursky–Tran \cite{CGT} and Li \cite{Li4}, who relaxed the assumption to $3$-positive ($3$-nonnegative) curvature operator of the second kind. In 1993, Kashiwada \cite{ka}  showed that Riemannian manifolds with harmonic curvature for which the curvature operator of the second kind is nonnegative (respectively, positive) are locally symmetric spaces (respectively, constant curvature spaces). Nienhaus, Petersen and Wink \cite{NP,NPWW} used the Bochner formula to show that $n$-dimensional compact Einstein manifolds with $k(<\frac{3n(n+2)}{2(n+4)})$-nonnegative curvature operators of the second kind are either rational homology spheres or flat. Dai-Fu \cite{DF} proved that  any compact Einstein manifold must have constant curvature if its curvature operator of the second kind is $k$-nonnegative, where $k=1$ for $4\leq n\leq7$, $k=2$ for $8\leq n\leq10$ and  $k=[\frac{n+2}{4}]$ for $n\geq11$.  Later, Dai-Fu-Yang \cite{DFY} proved a compact manifold with harmonic Weyl tensor and nonnegative curvature operator of the second kind is globally conformally equivalent to a space of positive constant curvature or is isometric to a flat manifold. 
	 
	Let $\lambda_{i}$ be the eigenvalues of the curvature operator, arranged in non-decreasing order. If $${\lambda }_1+{\lambda }_2+\cdots {\lambda }_{[k]}+(k-[k])\lambda _{[k]+1}> 0(\ge 0),$$ 
	we say that the curvature operator is $k$-positive(resp. nonnegative).
	In this paper,  we investigate Riemannian manifolds of dimension $n\ge3$ with harmonic curvature and $\frac{n(n+2)}{2(n+1)}$-nonnegative curvature operator of the second kind, and obtain the following Theorem~\ref{A}.
	\begin{thm}\label{A}
	Let $(M, g)$ be an $n(\ge 3)$-dimensional complete Riemannian manifold with harmonic curvature tensor. If the curvature operator of the second kind $\mathring{R}$ is $\frac{n(n+2)}{2(n+1)}$-nonnegative, then $M$ is an Einstein manifold.
	\end{thm}
	For Einstein manifolds, we establish the following theorem, which generalizes the result of \cite{DF}.
	
	\begin{thm}\label{B}
	Let $(M, g)$ be an $n(\ge 4)$-dimensional complete Einstein manifold. If the curvature operator of the second kind $\mathring{R}$ is $\frac{3n(n-1)^2(n+2)}{2(5n^3+3n^2-30n+16)}$-nonnegative, then $M$ must be of constant curvature.
    \end{thm}	
	Therefore, Based on Theorems~\ref{A} and ~\ref{B}, and the rigidity results for Einstein manifolds related to the curvature operator of the second kind — as established by Dai-Fu \cite{DF} and Li \cite{Li3}, we obtain the following corollary.
	\begin{corollary}
		Let $(M, g)$ be an $n(\ge 3)$-dimensional complete Riemannian manifold with harmonic curvature tensor, and $\mathring{R}\,$ be the curvature operator of the second kind. If\\ 
		(i) $3\le n\le 8$ and $\mathring{R}\,$ is $\frac{n(n+2)}{2(n+1)}$-nonnegative;\\
		(ii) $9\le n\le 15$ and $\mathring{R}\,$ is $4\frac12$-nonnegative;\\
		(iii) $n\ge 16$, $\mathring{R}\,$ is $\frac{3n(n-1)^2(n+2)}{2(5n^3+3n^2-30n+16)}$-nonnegative,\\
		then $M$ must be of constant curvature.
	\end{corollary}
\begin{remark}
Corollary 1.3 extends and enhances  Kashiwada's results \cite{ka} and Theorem 1.2 in \cite{DFY}. Corollary 1.3 can be regarded as a generalization of their results on Einstein manifolds \cite{{CGT},{DF},{Li3}}.
\end{remark}

Based on some existing results in this field and Corollary 1.3, combined with Li's conjecture that a closed $n$-dimensional Riemannian manifold with $(n+\frac{n-2}{n})$-positive curvature operator of the second kind is diffeomorphic to a spherical space form, it is natural to ask:

{\bf Question}  Are closed manifolds with harmonic curvature and $(n+\frac{n-2}{n})$-positive curvature operator of the second kind spherical space form?

\section{notation and lemmas}
	In a Riemannian manifold $(M^n,g)$, the Riemannian curvature tensor $R$  induces two symmetric linear operators $\hat{R}$ and $\mathring{R}$. Denote by $V$ be the tangent space of this Riemannian manifold, by ${\Lambda }^2V$  the space of skew symmetric $2$-tensor over $V$, and by $S_0^2(V)$  the space of trace-free symmetric $2$-tensor over $V$. The curvature operator of the first kind is defined by
\begin{align*}
	& \hat{R}:{\Lambda }^2V\to {\Lambda }^2V \\ 
	& \hat{R}(e_i\wedge e_j)=\frac{1}{2}\sum\limits_{k,l}{R_{ijkl}e_k\wedge e_l},
\end{align*}
and the curvature operator of the second kind is defined in \cite{CGT, NP} by
$$\mathring{R}={\Pr}_{s_0^2(V)}\circ \left.\bar{R} \right|_{S_0^2(V)}.$$
The operator $\mathring{R}$ can be regarded as the following curvature operator $\bar{R}$ and its image restricts on $S_0^2(V)$. The operator $\bar{R}$ act on $S^2(V)$ induced from Riemannian curvature tensor $R$ is defined by
$$\bar{R}(e_i\odot e_j)=\sum\limits_{k,l}{R_{kijl}e_k\odot e_l}.$$
Here $e_i\odot e_j=e_i\otimes e_j+e_j\otimes e_i$. So $\frac{1}{\sqrt{2}}\{e_i\odot e_j\}_{1\le i<j\le n},\frac{1}{2}\{e_i\odot e_i\}_{i=1,\dots,n}$ is an orthonormal basis for $S^2(V)$.
\begin{definition}
	Let $T^{(0,k)}(V)$ denote the space of $(0,k)$-tensor space on $V$. For $S\in S^2(V)$ and $T\in T^{(0,k)}(V),$ we define
	\begin{align*}
		& S:T^{(0,k)}(V)\to T^{(0,k)}(V) \\ 
		& (ST)(X_1,\cdots ,X_k)=\sum\limits_{i=1}^{k}{T(X_1,\cdots ,SX_i,\cdots ,X_k)} 
	\end{align*}
	and define $T^{S^2}\in T^{(0,k)}(V)\otimes S^2(V)$ by
	$$\left\langle T^{S^2}(X_1,\cdots ,X_k),S \right\rangle =(ST)(X_1,\cdots ,X_k).$$
\end{definition} 
Hence, if $\left\{ S^\alpha \right\}$ is an orthonormal basis for $S_0^2(V)$, then
\[T^{S_0^2}=\sum\limits_{\alpha=1}^{N}{S^\alpha T\otimes S^\alpha}.\]

To analyze the individual terms for Bochner formula, we apply the powerful weighted-sum calculus developed in \cite{NP}, which allows us to estimate finite weighted sums with nonnegative weights.
\begin{definition}(\cite[Definition 3.1]{NP})
	Let $\{\omega_{i}\}_{i=1}^{N}$ be the nonnegative weights of any finite weighted sums. Define
	\[
	\Omega=\max_{1\leq i\leq N}\omega_{i} \quad \mathrm{and} \quad  S=\sum_{i=1}^{N}\omega_{i}\,.
	\]
	We call $S$ the total weight and $\Omega$ the highest weight.
\end{definition} 
Let $\{\lambda_{i}\}_{i=1}^{N}$ be the eigenvalues of $\mathring{R}\,$. Following the notation in  \cite{NP}, we write $[\mathring{R},\Omega,S]$ to denote any finite weighted sums $\sum_{i=1}^{N}\omega_{i}\lambda_{i}$ in terms of $\{\lambda_{i}\}_{i=1}^{N}$ whose  weights satisfy highest weight $\Omega$ and total weight $S$. 
\begin{lemma}\label{Lem 3.4}(\cite[Lemma 3.4]{NP})
	Let  $\{\lambda_{i}\}_{i=1}^{N}$ denote the ordered eigenvalues of $\mathring{R}$. Then, for any integer $1\leq m\leq N$, we have
	$$[\mathring{R},\Omega,S]\geq(S-m\Omega)\lambda_{m+1}+\Omega\sum_{i=1}^{m}\lambda_{i}.$$
	Therefore, $[\mathring{R},\Omega,S]\ge 0$ if $\mathring{R}$ is $\frac{S}{\Omega}$-nonnegative.
	\end{lemma}

	\section{manifold with harmonic curvature}
	Let $M^n$ be an $n$-dimensional complete manifold with harmonic curvature, endowed with a standard Riemannian metric $\left\langle \text{,} \right\rangle$. Riemannian curvature tensor $R$ is divergence free if and only if its Ricci tensor is a Codazzi tensor, and in this case its scalar curvature is constant. Hence the traceless Ricci tensor is harmonic because its Ricci tensor is divergence free. Denote by $E$ its traceless Ricci tensor and by $s$ its scalar curvature. Then
	\begin{align}\label{laplacian}
		\frac{1}{2}\Delta \left| E \right|^2=\left| \nabla E \right|^2+\left\langle\Delta E,E \right\rangle.
	\end{align}
Since $E$ is harmonic, the Ricci identity gives
	\begin{align}\label{laplacian1}
		\left\langle\Delta E,E \right\rangle &=E_{ij}E_{ij,tt}=E_{ij}E_{it,jt}\nonumber\\
		&=E_{ij}(E_{it,tj}+R_{jtsi}E_{st}+R_{jtst}E_{is})\nonumber\\
		&=\left\langle \mathring{R}\,(E),E \right\rangle + R_{ij}E_{it}E_{jt},
	\end{align}		
	If $\{S_{\alpha}\}$ is an orthonormal eigenbasis for $\mathring{R}$ with corresponding eigenvalues $\{\lambda_{\alpha}\}$, then since $E$ is a traceless symmetric $(0,2)$-tensor, it can be expressed in the orthonormal basis $\left\{ S_{\alpha} \right\}$ as $E=\sum\nolimits_{\alpha}{E_{\alpha }S_{\alpha }}$. Then equation
	(\ref{laplacian1})  becomes
	 
	\begin{eqnarray}\label{laplacian2}
	\left\langle\Delta E,E \right\rangle=R_{ij}E_{it}E_{jt}+\sum\limits_{\alpha}{{\lambda}_{\alpha}E_{\alpha}^2}.
	\end{eqnarray}
	According to \cite[Example 3.2]{NP}, the scalar curvature $s$ of $M$ satisfies
	\begin{align}
	s\ge \frac{2n}{n+2}[\mathring{R},1,\frac{(n-1)(n+2)}{2}],\label{scalar}
\end{align}
	and from \cite[Lemma 3.14]{NP}, the Ricci curvature $Ric$ of $M$ satisfies
	\begin{align*}
		\Ric \geq \frac{n-1}{n+1} \left[ \mathring{R}, 1, n \right] + \frac{1}{n(n+1)}s.
	\end{align*}
	Combining (\ref{scalar}) with \cite[Lemma 3.3(c)]{NP}, we obtain
		\begin{align}
		\Ric &\geq \frac{n-1}{n+1} \left[ \mathring{R}, 1, n \right] + \frac{2}{(n+1)(n+2)}[\mathring{R},1,\frac{(n-1)(n+2)}{2}]\nonumber\\
		&\ge \left[ \mathring{R}, \frac{n}{n+2}, n-1 \right].\label{estimate2}
	\end{align}	
Furthermore, since
$$\sum\limits_{\alpha }{E_{\alpha }^2}=\left| E \right|^2\quad\text{and}\quad E_{\alpha}^2\le \left| E \right|^2,$$
the weight principle in \cite[Theorem 3.6]{NP} yields 
\begin{align}
\left\langle \mathring{R}\,(E),E \right\rangle \ge \left[\mathring{R},1,1 \right]\left| E \right|^2.\label{estimate3}
\end{align}	

\begin{proposition}\label{2.3}
	If $E$ is a traceless symmetric 2-tensor, then
	 $$\left\langle\Delta E,E \right\rangle \ge \left[\mathring{R},\frac{2(n+1)}{n+2},n\right]\left|E\right|^2.$$
\end{proposition}
\begin{proof}
Combining equation (\ref{laplacian2}) with (\ref{estimate2}),(\ref{estimate3}) and \cite[Lemma 3.3(c)]{NP} yields 
\begin{align*}
	 \left\langle\Delta E,E \right\rangle \ge& \left[ \mathring{R}, \frac{n}{n+2}, n-1 \right]\left| E \right|^2+\left[\mathring{R},1,1 \right]\left| E \right|^2 \\ 
	=&\left[\mathring{R},\frac{2(n+1)}{n+2},n\right]\left|E\right|^2.  
\end{align*}
\end{proof}

\begin{proof}[\textbf{Proof of Theorem~\ref{A}}] 
If $s=0$, then $M$ is flat.	Otherwise, $s>0$,   i.e.,   $M$ is non-flat. By part $(2)$ of Proposition 4.1 in \cite{Li4}, we have that $Ric(M)\geq\frac{s}{n(n+1)}>0$. So $M$ is compact by Myers Theorem. Let  $\tilde{M}$ be the universal cover
of  $M$, equipped with the lifted metric  $\tilde{g}$. Then $(\tilde{M}, \tilde{g})$ also has $\frac{n(n+2)}{2(n+1)}$-nonnegative curvature
operator of the second kind, and nonnegative Ricci curvature. According to the Cheeger-Gromoll splitting theorem \cite{CG71},  $\tilde{M}$ is isometric to a
product of the form $N^{n-k}\times R^k$, where $N$ is compact. By Theorem 1.8 in \cite{Li4}, we know
that $\tilde{M}$  is locally irreducible, which forces  $k=0$. Hence $\tilde{M}$ is compact. By Proposition~\ref{2.3} and  Lemma~\ref{Lem 3.4},  we have 
$$\left\langle\Delta E,E \right\rangle \ge \left[\mathring{R},\frac{2(n+1)}{n+2},n\right]\left|E\right|^2\ge 0.$$
It follows from (\ref{laplacian}) that the trace-less Ricci tensor $E$ is parallel. Hence $\tilde{M}$ is  Ricci parallel and an Einstein 
manifold.  Consequently, $M$ is an Einstein 
manifold.
\end{proof}

\section{Einstein manifold}
This section derives a refined Bochner formula for Einstein manifolds, leading to a proof of Theorem~\ref{B}. We begin by establishing a Bochner formula for the curvature tensor.
\begin{proposition}
Let $(M,g)$ be an $n$-dimensional Einstein manifold, then the curvature operator $R$ of $M$ satisfies
\begin{equation}\label{Bochner}
	3\left\langle \Delta R,R \right\rangle =\sum\limits_{\alpha}{{\lambda }_\alpha\left| S^{\alpha}W \right|^2} +16\sum\limits_{\alpha}{{\lambda}_{\alpha}({\lambda }_{\alpha}-\frac{s}{n(n-1)})^2}+\frac{4(n-4)}{n^2}s\left| W \right|^2.
\end{equation}
\end{proposition}
\begin{proof}
	The following Bochner formula is established by Dai and Fu in \cite[Proposition 3.4]{DF}.
\begin{align*}
	3\left\langle \Delta R,R \right\rangle =&\sum\limits_{\alpha}{{\lambda }_{\alpha}\left| S^{\alpha}W \right|^2}+8\left( \frac{-n^3+6n^2+12n-8}{3n^4(n-1)^2} \right)s^3 \\ 
	& +8\left( \frac{2n^2-22n+8}{3n^2(n-1)} \right)s\sum\limits_{\alpha}{\lambda _{\alpha}^2}+16\sum\limits_{\alpha}{\lambda _{\alpha}^3}. 
\end{align*}
From \cite[Lemma 3.3]{DF}, we know 
\begin{eqnarray}\label{trace}
	tr(\mathring{R})&=&\frac{n+2}{2n}s,\nonumber\\
	tr(\mathring{R}^2)&=&\frac{3}{4}|R|^2-\frac{s^2}{n^2}=\frac{3}{4}|W|^2+\frac{n+2}{2n^2(n-1)}s^2.
\end{eqnarray}
Combining the Bochner formula above with  (\ref{trace}), we obtain 
\begin{equation*}
 3\left\langle \Delta R,R \right\rangle =\sum\limits_{\alpha}{{\lambda }_\alpha\left| S^{\alpha}W \right|^2} +16\sum\limits_{\alpha}{{\lambda}_{\alpha}({\lambda }_{\alpha}-\frac{s}{n(n-1)})^2}+\frac{4(n-4)}{n^2}s\left| W \right|^2.
\end{equation*}
\end{proof}
First, let's estimate the first term on the right-hand side of (\ref{Bochner}). According to the calculations in \cite{DF}, we have
\begin{eqnarray*}
	\sum\limits_{\alpha=1}^{N}{\left|S^{\alpha}W \right|^2}=\frac{2(n^2+n-8)}{n}{\left| W \right|^2},
\end{eqnarray*}
and $$\left| S^{\alpha}W \right|^2 \le\frac{8(n-2)}{n}\left| W \right|^2.$$ 
Hence
\begin{equation}\label{estimate}
	\sum\limits_{\alpha}{{\lambda }_{\alpha}\left| S^{\alpha}W \right|^2}\ge [\mathring{R},\frac{8(n-2)}{n},\frac{2(n^2+n-8)}{n}]|W|^2.
\end{equation}

Second, we estimate the second term on the right-hand side of (\ref{Bochner}). Set $\beta_\alpha={\lambda }_\alpha-\frac{s}{n(n-1)}$. From (\ref{trace}), we obtain
\[\left\{ \begin{aligned}
	& \sum\limits_{\alpha}{{\beta }_{\alpha}}=0, \\ 
	& \sum\limits_{\alpha}{\beta _{\alpha}^2}=\frac{3}{4}\left| W \right|^2. \\ 
\end{aligned} \right.\]
For any fixed $k$, the Cauchy–Schwarz inequality give
\[(\sum_{\alpha\ne k}\beta_\alpha)^2\le(N-1)\sum_{\alpha\ne k}{\beta_{\alpha}^2},\]
and hence
\[\beta_k^2\le(N-1)[\frac{3}{4}|W|^2-\beta_k^2].\]
This implies that
\[\beta_k^2\le \frac{3}{4}\frac{N-1}{N}|W|^2.\]
Therefore, we obtain
\begin{equation}\label{estimate1}
	\sum\limits_{i}{{\lambda}_i({\lambda }_i-\frac{s}{n(n-1)})^2}\ge \frac{3}{4}[\mathring{R},\frac{n^2+n-4}{(n-1)(n+2)},1]|W|^2.
\end{equation}

\begin{proposition}\label{2.4}
	Let $M$ be an $n$-dimensional Einstein manifold with $n\ge 4$. Then the curvature operator $R$ of $M$ satisfies
	$$3\left\langle\Delta R, R \right\rangle \ge \left[\mathring{R},\frac{4(5n^3+3n^2-30n+16)}{n(n-1)(n+2)},6(n-1)\right]\left|W\right|^2.$$
\end{proposition}
\begin{proof}
	 Combining equation (\ref{Bochner}) with (\ref{scalar}),(\ref{estimate}),(\ref{estimate1}) and \cite[Lemma 3.3(c)]{NP} yields
\begin{align*}
	& 3\left\langle \Delta R, R \right\rangle =\left\langle \mathring{R}\,(W^{S_{0}^2}),W^{S_{0}^2} \right\rangle +16\sum\limits_{i}{{\lambda}_i({\lambda }_i-\frac{s}{n(n-1)})^2}+\frac{4(n-4)}{n^2}s\left| W \right|^2 \\ 
	& \ge [\mathring{R},\frac{8(n-2)}{n},\frac{2(n^2+n-8)}{n}]\left| W \right|^2+12[\mathring{R},\frac{n^2+n-4}{(n-1)(n+2)},1]\left| W \right|^2+\frac{8(n-4)}{n(n+2)}[\mathring{R},1,\frac{(n-1)(n+2)}{2}]\left| W \right|^2 \\ 
	& \ge [\mathring{R},\frac{8(n-2)}{n}+\frac{n^2+n-4}{(n-1)(n+2)}+\frac{8(n-4)}{n(n+2)},\frac{2(n^2+n-8)}{n}+12+\frac{4(n-4)(n-1)}{n}]\left| W \right|^2 \\ 
	& =[\mathring{R},\frac{4(5n^3+3n^2-30n+16)}{n(n-1)(n+2)},6(n-1)]\left| W \right|^2.
\end{align*}
\end{proof}
\begin{proof}[\textbf{Proof of Theorem~\ref{B}}.]
If $s=0$, then $M$ is flat since $\mathring{R}$ is $\frac{3n(n-1)^2(n+2)}{2(5n^3+3n^2-30n+16)}$-nonnegative.	Otherwise, $s>0$,  and $M$ is compact by Myers Theorem for $M$ is Einstein.
	
	If $\mathring{R}$ is $\frac{3n(n-1)^2(n+2)}{2(5n^3+3n^2-30n+16)}$-nonnegative, then by using Proposition~\ref{2.4} and  Lemma~\ref{Lem 3.4}, we get
	\[\frac{1}{2}\Delta\left| R \right|^2=\left| \nabla R \right|^2+\left\langle\Delta R,R\right\rangle \ge 0.\]
	By the maximum principle, we have $ \nabla R=0$. It then follows that $M$ is a locally symmetric space.

On the other hand, It follows from Theorem B of \cite{NP} that	M is a rational homology sphere.	Wolf \cite[Theorem 1]{Wo} provides a complete classification of compact symmetric spaces that are rational homology spheres. Aside from spheres themselves, the only simply connected example is the space $SU(3)/SO(3)$, which is shown in \cite[Example 4.5]{NP} to be 9-positive. It then follows that the manifold must have constant curvature. This completes the proof of Theorem~\ref{B}.
\end{proof}

\begin{thebibliography}{NPWW23}
\bibitem[BG24]{BG}
Bettiol, Renato G and Goodman, McFeely Jackson. \emph{Curvature operators and rational cobordism.}
Advances in Mathematics, \textbf{458}:109995, 2024.

\bibitem[CG71]{CG71}
Jeff Cheeger and Detlef Gromoll. \emph{The splitting theorem for manifolds of nonnegative Ricci curvature.}
Journal of Differential Geometry, \textbf{6}(1):119 – 128, 1971.

\bibitem[CGT23]{CGT} 
Xiaodong Cao, Matthew J Gursky, and Hung Tran. \emph{Curvature of the second kind and a conjecture of Nishikawa.} Commentarii Mathematici Helvetici, \textbf{98}(1):195–216, 2023.

\bibitem[CMR24]{CMR} 
Giulio Colombo, Marco Mariani, and Marco Rigoli. \emph{Tachibana-type theorems on complete manifolds.}
Annali della Scuola Normale Superiore di Pisa - Classe di Scienze, \textbf{25}(2):1033–1083, 2024.

\bibitem[DF24]{DF}
 Zhi-Lin Dai and Hai-Ping Fu. \emph{Einstein manifolds and curvature operator of the second kind.}
Calculus of Variations and Partial Differential Equations, \textbf{63}(2):53, 2024.

\bibitem[DFY24]{DFY} 
Zhi-Lin Dai, Hai-Ping Fu, and Deng-Yun Yang. \emph{Manifolds with harmonic Weyl tensor and nonnegative
	curvature operator of the second kind.} Journal of Geometry and Physics, \textbf{195}:105040,
2024.

\bibitem[Kas93]{ka} 
Toyoko Kashiwada. \emph{On the curvature operator of the second kind.} Natural Science Report, Ochanomizu
University, \textbf{44}(2):69–73, 1993.

\bibitem[Li22]{Li3} 
Xiaolong Li. \emph{Manifolds with $4\frac{1}{2}$-positive curvature operator of the second kind.} The Journal of
Geometric Analysis, \textbf{32}(11):281, 2022.

\bibitem[Li24]{Li4} 
Xiaolong Li. \emph{Manifolds with nonnegative curvature operator of the second kind.} Communications
in Contemporary Mathematics, \textbf{26}(03):2350003, 2024.

\bibitem[Nis86]{Nis} 
Seiki Nishikawa. \emph{On deformation of Riemannian metrics and manifolds with positive curvature
	operator.} In Curvature and Topology of Riemannian Manifolds (Katata, 1985), 202–211, 1986.

\bibitem[NPW23]{NP}
Jan Nienhaus, Peter Petersen, and Matthias Wink. \emph{Betti numbers and the curvature operator of
	the second kind.} Journal of the London Mathematical Society, \textbf{108}(4):1642–1668, 2023.

\bibitem[NPWW23]{NPWW}
 Jan Nienhaus, Peter Petersen, Matthias Wink, and William Wylie. \emph{Holonomy restrictions from
	the curvature operator of the second kind.} Differential Geometry and its Applications, \textbf{88}:102010,
2023.

\bibitem[PW21]{PW} 
Peter Petersen and Matthias Wink. \emph{New curvature conditions for the Bochner technique.} Inventiones
mathematicae, \textbf{224}(1):33–54, 2021.

\bibitem[PW22]{PW1} 
Peter Petersen and Matthias Wink.\emph{ Tachibana-type theorems and special holonomy.} Annals of
Global Analysis and Geometry, \textbf{61}:847–868, 2022.

\bibitem[Tac74]{Tac}
 Shun-ichi Tachibana. \emph{A theorem on Riemannian manifolds of positive curvature operator.} Proceedings
of the Japan Academy, \textbf{50}(4):301 – 302, 1974.

\bibitem[Tra17]{tran} 
Hung Tran. \emph{On closed manifolds with harmonic Weyl curvature.} Advances in Mathematics,
\textbf{322}:861–891, 2017.

\bibitem[Wol69]{Wo}
Joseph Wolf. \emph{Symmetric spaces which are real cohomology spheres.} Journal of Differential Geometry,
\textbf{3}(1-2):59–68, 1969.
\end{thebibliography}

\end{document}